\documentclass[11pt]{article}
\usepackage{amsmath,amsthm,amsfonts,latexsym,amssymb,enumerate,color}
\usepackage{graphicx}
\usepackage{mathrsfs}

\newtheorem{thm}{Theorem}[section]
\newtheorem{cor}[thm]{Corollary}
\newtheorem{prop}[thm]{Proposition}

\newtheorem{rem}[thm]{Remark}

\def\CC{\mathbb C}

\def\DD{\mathbb D}

\def\G{\mathcal G}

\def\GN{\mathcal N}

\def\GR{\mathcal R}
\def\RR{\mathbb R}
\def\TT{\mathbb T}

\def\ZZ{\mathbb Z}
\def\beginpf{\begin{proof}}
\def\endpf{\end{proof}}
\def\beq{\begin{equation}}
\def\eeq{\end{equation}}
\def\ol{\overline}
\def\til{\tilde}

\def\wt{\widetilde}

\def\wX{\widetilde X}
\def\wY{\widetilde Y}
\def\ker{\mathop{\rm ker}\nolimits}
\def\ran{\mathop{\rm ran}\nolimits}
\newcommand{\simast}{%
    \ensuremath{%
       \stackrel{\mathsf{\ast}}{\sim}}}
\def\spam{\mathop{\rm span}\nolimits}   

\def\ind{\mathop{\rm ind}\nolimits}

\def\CP{{\rm CP}}
\def\GCD{{\rm GCD}}
\def\essinf{\mathop{\rm ess \, inf}}
\def\clos{\mathop{\rm clos}\nolimits}

\begin{document}

\title{Equivalence after extension for general Toeplitz operators}
\author{M.~Cristina C\^amara\thanks{Center for Mathematical Analysis, Geometry and Dynamical Systems,
Instituto Superior T\'ecnico, Universidade de Lisboa, 
Av. Rovisco Pais,  1049-001 Lisboa, Portugal.   {\tt cristina.camara@tecnico.ulisboa.pt}}\quad  and\quad
Jonathan R.~Partington\thanks{School of Mathematics, University of Leeds, Leeds LS2~9JT, U.K. {\tt j.r.partington@leeds.ac.uk}}}

\date{}

\maketitle
 
\centerline{\em In memory of Franciszek Hugon Szafraniec}

\medskip

\begin{abstract}
This paper reviews the subject of equivalence after extension in the context of Toeplitz operators,
truncated Toeplitz operators, and further generalizations including dual and multiband truncated Toeplitz
operators. Paired operators are shown to play a key role here. Properties under investigation
by these methods
include the invertibility and Fredholm properties of generalized Toeplitz operators, with
new results showing how the kernels of such operators can often be fully described
given only limited information.
 \end{abstract}


\section{Introduction}

In this article we revisit the concept of Equivalence After Extension (EAE)
for operators on Banach spaces --   in this case, Hilbert spaces -- and its relations with various types of compressions of multiplication operator (which we call {\em general Toeplitz operators\/}).
We explain how these relations allow us not only to study various
properties that may be difficult to see otherwise, but also to obtain a better understanding of the concept
itself. In addition, we note how natural isomorphisms of kernels can suggest EAE, and we obtain new
insights and new properties not previously mentioned in the literature.

Let $X,\wX,Y,\wY$ be Banach spaces and let $T:X \to \wX$ and $S: Y \to \wY$ be bounded operators. We say that $T$ and $S$ are {\em equivalent\/} operators if and only if there exist operators $E$ and $F$, invertible, such that
\[
T=ESF \qquad (\hbox{written} \quad T \sim S).
\]
If $T$ and $S$ are equivalent, then they are simultaneously the zero operator or not,
invertible or not, Fredholm or not (with the same Fredholm defect numbers in
the first case), with isomorphic kernels and isomorphic ranges.
In the two latter cases the isomorphisms are immediately given by 
\[ 
\ker T=F^{-1}\ker S, \qquad \ran(T)= E \ran S. 
\]
However, equivalent operators do not in general have the same spectra: for example $I$, the identity
operator, is clearly equivalent to $2I$.

This notion of equivalence has been extensively  used in the study of the solvability of
Wiener--Hopf and singular integral equations; obtaining such an equivalence has been a strong motivation to study Wiener--Hopf factorisation and develop explicit methods to obtain it.\\

To give a   simple example of this, let $\TT$ denote the unit circle in the complex plane and
let $C^\mu(\TT)$, for $0< \mu < 1$, 
denote the algebra of H\"older-continuous functions with exponent $\mu$; that is
$G \in C^\mu(\TT)$ if and only if there is an $A>0$ such that
\[
 |G(z_1)-G(z_2)| \le A |z_1-z_2|^\mu \qquad (z_1,z_2 \in \TT).
\]
Suppose that $G \in C^\mu(\TT)$ is invertible, so that $G(z) \ne 0$ for all $z \in \TT$
(we write this $G \in \G C^\mu(\TT)$). Then $G$ admits a factorisation, known as a
Wiener--Hopf factorisation, as follows:
\beq\label{eq:4}
G= G_- z^k G_+
\eeq
with $k \in \ZZ$, $G^{\pm 1}_- \in \ol{H^\infty} \cap C^\mu(\TT)$, and
$G^{\pm 1}_+ \in H^\infty \cap C^\mu(\TT)$.
Here $k$ is the winding number of $G$ around the origin in the complex plane.\\

Now denote by $P^+$ the orthogonal projection from $L^2:=L^2(\TT)$ onto $H^2_+:= H^2(\DD)$ and define the {\em Toeplitz operator with symbol $G$\/} by
\beq \label{eq:5}
T_G = P^+ G P^+_{| H^2_+}.
\eeq
The factorization \eqref{eq:4}
induces the operator factorization
\[
T_G = T_{G_-} T_{z^k} T_{G_+},
\] 
where $T_{G_-}$ and $T_{G_+}$ are invertible, so we have that
\[
T_G \sim T_{z^k},
\]
thus reducing the study of many properties of $T_G$ to those of a very simple Toeplitz operator $T_{z^k}$.\\

The concept of equivalent operators was generalised by Bart and Tseka\-novski\u\i~\cite{BTsk} in 1992. We say that two operators $T: X \to \wX$ and $S:Y \to \wY$ are
{\em equivalent after extension\/} (EAE) if and only if there exist Banach spaces $X_0$, $Y_0$
and invertible operators $E,F$ such that
\beq\label{eq:8}
\begin{pmatrix}
T & 0 \\ 0 & I_{X_0}
\end{pmatrix} = 
E
\begin{pmatrix}
S & 0 \\ 0 & I_{Y_0}
\end{pmatrix}
F.
\eeq
The spaces $X_0$ and $Y_0$ are called {\em extension spaces}. Then we write $T \simast S$.
That is,
\[ T \simast S \iff \begin{pmatrix}
T & 0 \\ 0 & I_{X_0}
\end{pmatrix} \sim
\begin{pmatrix}
S & 0 \\ 0 & I_{Y_0}
\end{pmatrix}.
\]
Some shared properties of operators that are EAE are the following:
\begin{itemize}
\item $\ker T \cong \ker S$;
\item $\ran T$ is closed if and only if $\ran S$ is closed, and in that case\\
$\wX/\ran T \cong \wY/\ran S$;
\item if one of the operators is left/right invertible, then so is the other;
\item if $T$ is Fredholm  (that is, if $T$  has closed range and
the spaces   
$\ker T$ and $\wX/\ran T$ are finite-dimensional), then  $S$ is also Fredholm, with $\dim \ker T=\dim \ker S$
and $\dim \wX/\ran T = \dim  \wY/\ran S$.
\end{itemize}
Some other properties are not necessarily shared.
For instance, one operator may be the zero operator, when the other is not; 
moreover, they do not necessarily share the same spectrum.

However, one may reformulate the questions of the zero operator or 
the spectrum in terms of kernels and invertibility, in such a way that one still obtains
answers for one operator from the study of another operator that is EAE. We give examples later.

Useful as the notion of EAE may be, behind it lie several difficult questions.\\

\noindent {\bf Question 1}. How do we show that two operators are EAE?\\

In general we have to construct the operators $E$ and $F$ in \eqref{eq:8} and choose the spaces $X_0$ and $Y_0$. Moreover,
given an operator $T$ whose properties one wants to study, many operators can be chosen that
are EAE to $T$. The choice is useful if studying the corresponding properties
of $S$ is simpler that the original task for $T$. So there are two more natural questions:\\

\noindent {\bf Question 2.} How do we find a {\em useful\/} EAE relation?\\

\noindent {\bf Question 3}. Which extension spaces should we choose?\\

In what follows we take advantage of the strong connection that must exist between the solutions of the two equations
\[
Tf=g, \qquad S\phi=\psi,
\]
and, in particular,
\[
Tf=0, \qquad S\phi=0,
\]
to suggest answers to the questions above in various examples.

\section{Toeplitz operators on $H^2_+$}

Recall that the Toeplitz operator $T_G: H^2_+ \to H^2_+$ with symbol $G \in L^\infty$
was defined  in \eqref{eq:5}.
Its kernel consists of the functions $\phi_+ \in H^2_+$ such that 
$P^+ G \phi_+=0$; that is, $G\phi_+ \in H^2_-:= L^2 \ominus H^2_+$. Thus,
for any $\phi_+ \in H^2_+$ we have
\begin{align}\label{eq:18}
T_G \phi_+ = 0 & \iff G\phi_+\in H^2_-\nonumber
\\
& \iff G\phi_+ + \phi_- = 0 \quad \hbox{for some } \phi_- \in H^2_- \nonumber \\
& \iff (GP^+ + P^-)\phi=0 \quad \hbox{with } \phi \in L^2 \hbox{ and } P^\pm\phi=\phi_\pm.
\end{align}
The operator $GP^+ + P^-$ is called a {\em paired operator}. These are operators
of the form 
\[
S_{A,B} = AP^+ + B P^-,
\]
where $A$ and $B$ are bounded operators on $L^2$. 
For a recent study of their properties we refer to \cite{CGP24,CP24}.
In this case
we take $A$ to be multiplication by $G$ and $B=I$, the identity. From
\eqref{eq:18} we see that
\beq\label{eq:19}
\phi_+ \in \ker T_G \iff \phi \in \ker (GP^++P^-), \hbox{ and } P^+\phi=\phi_+ 
\eeq
and
\[
P^+: \ker (GP^++P^-) \to \ker T_G
\]
is an isomorphism with inverse defined by
\[
\phi_+ \mapsto \phi_+ - G \phi_+ \quad \hbox{for } \phi_+ \in \ker T_G.
\]
From \eqref{eq:19} and $\ker T_G=P^+ \ker (GP^++P^-)$ we get a natural candidate for
an operator that is EAE to $T_G$, as well as a 
hint of which the unknown spaces should be. Indeed, we have the following:
\begin{thm}\label{thm:2.1aug3}
\beq\label{eq:22}
GP^++P^- \simast T_G
\eeq
\end{thm}
\begin{proof}
This follows since
\[
\begin{pmatrix}
T_G & 0 \\ 0 & I_{H^2_-}
\end{pmatrix}
=
\begin{pmatrix}
P^+ & -P^- \\ P^- & P^+ 
\end{pmatrix}
\begin{pmatrix}
GP^++P^- & 0 \\ 0 & I_{\{0\}}
\end{pmatrix}
\begin{pmatrix}
P^+ -P^-GP^+ & P^- \\ 0 & I_{\{0\}}
\end{pmatrix},
\]
where the factors on the left and right of the right-hand side are invertible operators.
\end{proof}
Regarding Question 2, the relation \eqref{eq:22} is useful because there is a well-developed
factorisation theory to study properties such as invertibility and Fredholmness of paired operators.

To illustrate how we can use EAE to study properties of one operator that are not shared by 
the other, but which can be reformulated in an appropriate way, consider the following:
{\em when is a Toeplitz operator the zero operator?}

Note that
this property is not shared by the associated paired operator on $L^2$, since
for any $G \in L^\infty$,
the paired operator $GP^+ + P^-$ is never the zero operator (consider its action on 
a function in $H^2_-$). So, one can ask instead, when is
$P_+ \ker (GP^+ + P^-) = H^2_+$?

We start by showing that $P_+ \ker (GP^+ + P^-)$ must be a model space if functions of
a certain type belong to that kernel.

\begin{prop}\label{prop:2.2jul30}
Let $\theta$ be an inner function and let $G \in L^\infty\setminus\{0\}$.
If a function of the form $f=\theta+f_-$, with $f_- \in H^2_-$, belongs to 
$\ker (GP^+ + P^-)$, then we must have $\ol G \in z\theta H^\infty$ and,
if $\ol G = \theta z G_i G_o$, with $G_i$ inner
and $G_o$ outer, in an inner--outer factorization, then
\[
P_+ \ker (GP^+ + P^-) = \ker T_G = K_{z\theta G_i}.
\]
\end{prop}
\begin{proof}
If $G\theta + f_-=0$ then 
\[
G=-\ol\theta f_- \in  \ol\theta H^2_-\cap L^\infty  = \ol z\ol\theta \ol{H^2_+} \cap L^\infty = \ol z\ol\theta \ol{H^\infty}.
\]
Writing $\ol G=\theta z G_i G_o$, we have that, for any $f \in \ker (GP^+ + P^-)$ with $P^\pm f = f_\pm$,
\[
\ol z\ol \theta \ol{G_i} \ol{G_o} f_+=-f_- \iff \ol z\ol\theta \ol{G_i} f_+=- \frac{f_-}{\ol{G_o}} \in \ol z \ol{\GN^+} \cap L^2= H^2_-.
\]
So, $f_+ \in \ker T_G \implies f_+ \in K_{z\theta G_i}$.

Conversely, if $f_+ \in \ker T_{\ol z \ol\theta \ol{G_i}}= K_{z\theta G_i}$, then
\[
Gf_+ = \ol z \ol\theta \ol{G_i} \ol{G_o} f_+ = \ol{G_o} (\underbrace{\ol z\ol\theta \ol{G_i} f_+}_{\in H^2_-}) \in H^2_-,
\]
so $f_+ \in \ker T_G$.
\end{proof}

Taking, in particular, $f=1+f_-$ with $f_- \in H^2_-$, we have:
\begin{cor}
Let $g \ne 0$. Then $1 \in \ker T_G \iff \ker T_G = K_{z\alpha}$ for some inner function $\alpha$. 
\end{cor}
Thus we can present the following well-known result in the light of the notion of EAE.
\begin{cor}
$T_G=0 \iff G=0$.
\end{cor}

It is also clear that, in general, two operators that are EAE do not have identical spectra. 
Thus, to study the spectrum of a Toeplitz operator by means of EAE one must
reformulate the question in terms of invertibility:
when is $T_G-\lambda I_{H^2_+} = T_{G-\lambda}$ invertible? 
This is equivalent to asking : when is $(G-\lambda)P^+ + P^-$ invertible?

The answer to this last question is known (see, e.g. \cite{MP}): it is when $G-\lambda$ admits a 
canonical generalised factorisation relative to $L^2$.
If $G$ is continuous on $\TT$, this can be expressed by saying
that $\lambda$ does not belong to the image of $G$ in the complex plane
and the index $\ind_\lambda G=0$.

By using the approach expressed in \eqref{eq:18} we can also present some invariance properties
of Toeplitz kernels, meaning that if a function $f$ belongs to a 
particular kernel then a subspace of $H^2_+$, determined by $f$, must also be contained in
that kernel.

Recall that
a {\em nearly $S^*$-invariant subspace\/} $M \subset H^2_+$ is one such that
if $f \in M$ and $f(0) =0$ then $f/z \in M$ (that is, one can divide out zeros). This is
a property possessed by model spaces, which are invariant under the backward shift $S^*$,
and indeed by kernels of general Toeplitz operators.
We shall abbreviate this to {\em nearly invariant\/} in what follows.

Following work of Hitt~\cite{hitt} and Hayashi~\cite{hayashi86,hayashi90} it is known that such spaces $M$ have the structure
$M=h_+K$, where $K$ is an $S^*$-invariant subspace (that is, a model space $K_\theta$
in the nontrivial case) and $h_+ \in M$ is an isometric multiplier from $K$ onto $M$.

Toeplitz kernels are nearly $S^*$-invariant subspaces of $H^2_+$, and
in fact they are even nearly $\ol{H^\infty}$-invariant: 
in other words, if $f_+ \in \ker T_g$ and
$h \in H^\infty$ satisfies $\ol h f_+  \in H^2_+$, then $\ol h f_+  \in \ker T_g$.
For if $g f_+ = f_- \in H^2_-$ then $g \ol h f_+= \ol h f_-$, which also lies in $H^2_-$.

For an inner function $\theta \in H^\infty$ we define the subspace
\[
K^\infty_{ \theta}= \{h \in H^\infty: \ol\theta h \in \ol{H^2_+}\}=\{h \in H^\infty: \theta \ol h \in H^2_+ \}.
\]
Note that for such an $h$ we have $\ol z \ol \theta h\in H^2_-$ and so $h \in \ker T_{\ol z\ol\theta}=K_{z\theta}$.
From this we see easily that $K^\infty_{\theta} = K_{z\theta} \cap H^\infty$.

\begin{prop}\label{prop:2.4jul30}
If  $f_+ \in \ker T_g$ for some $g \in L^\infty$, and $f_+=f_{i+}f_{o+}$ is its inner--outer factorization
then $\ol{K^\infty _{{f_{i+}}} }f_+ \subset \ker T_g$.
\end{prop}
\begin{proof}
If
$h \in K^\infty _{{f_{i+}}}$ and $g f_+ = f_- \in H^2_-$ then 
\[
\ol h f_+ = \ol h f_{i+} f_{o+} \in H^1_+ \cap L^2=H^2_+
\]
 and
\[
 g \ol h f_+ = f_- \ol h \in H^2_-.
\]
\end{proof}

In particular, we see that the property of near $S^*$-invariance for Toeplitz kernels can be expressed as in Proposition~\ref{prop:2.4jul30}. Indeed, for $f_+ \in \ker T_g$, the condition $f_+(0)=0$ is equivalent to
$z$ dividing $f_{i+}$, and $\ol z \in K^\infty_{f_{i+}}$, so $\ol z f_+ \in \ker T_g$.

Another form of (near) invariance arises if we define $R_\TT$ to denote the space of rational functions whose
only poles lie on $\TT$; in particular, they are proper, in the sense that they have a finite limit at $\infty$.
Now for $u$ an outer function we define
\[ 
 K'_u = \{ r \in R_\TT: \ol u r \in \ol{H^2_+} \}.
\]
\begin{prop}
If  $f_+ \in \ker T_g$ for some $g \in L^\infty$, and $f_+=f_{i+}f_{o+}$ is its inner--outer factorization
then $\ol{K'_{f_{o+}}} f_+ \subset \ker T_g$.
\end{prop}
\begin{proof}
For $r \in K'_{f_{o+}}$ we have 
$\ol r f_+ = \ol r f_{i+}f_{o+} \in H^2_+$.

Also, with $g f_+=f_-$ as before, we have
$g \ol r f_+= \ol r f_- \in H^2_-$. This last assertion follows because $\ol z r  \ol{f_-} \in L^2 \cap \GN^+ = H^2_+$,
where $\GN^+$ is the Smirnov class (the class of functions that can be expressed as the ratio of two $H^2_+$ functions with the denominator outer).
\end{proof}

\begin{rem}{\rm
The situation for Hankel operators is less developed than
for Toeplitz operators. One commonly used definition of a Hankel operator
$\Gamma_G: H^2_+ \to H^2_-$ is the following:
\[
\Gamma_G f = P^- G P^+ f, \qquad f \in H^2_+.
\]
In order to obtain a unitarily equivalent operator on $H^2_+$ we may use
 the linear inversion involution, $J: z^n \mapsto z^{-n-1}$ $(n \in \ZZ)$, which exchanges $H^2_+$ and $H^2_-$, defining
 \[
 \wt \Gamma_G f = P^+ J G P^+ f, \qquad f \in H^2_+.
 \]
 We can link these to paired operators
since $\Gamma_G \phi_+=0 \iff P^- G\phi_+=0 \iff JG \phi_+= -\phi_-$ for some $\phi_- \in H^2_-$,
and thus, with $\phi=\phi_+ + \phi_-$ we have
\beq\label{eq:a23}
(JG P^+ + P^-)\phi=0.
\eeq
Conversely, if $\phi$ satisfies \eqref{eq:a23}, we see that $\phi_+ \in \ker \Gamma_G$.
}
\end{rem}

\section{(Asymmetric) truncated Toeplitz operators}
\label{sec:3}

Recall that for $\theta$ an inner function, the model spaces $K_\theta =H^2_+ \ominus \theta H^2_+$  are the nontrivial invariant subspaces for the
backward shift $T_{\ol z}$.
Now for $G \in L^\infty$ and $\theta, \alpha$ inner we define the {\em asymmetric
truncated Toeplitz operator  (ATTO) with symbol $G$}
\[
A^{\theta,\alpha}_G: K_\theta \to K_\alpha, \qquad A^{\theta,\alpha}_G = P_\alpha G {P_\theta}_{|K_\theta},
\]
where $P_\alpha$ and $P_\theta$ are the orthogonal projections from $L^2$ onto
$K_\alpha$ and $K_\theta$ respectively. 
When $\theta=\alpha$ we have the {\em truncated Toeplitz operators\/} (TTO)
of Sarason \cite{sarason07}, namely
\[
A^\theta_G: K_\theta \to K_\theta, \qquad A^\theta_G=P_\theta G {P_\theta}_{|K_\theta}.
\]
One can also define ATTO with symbols $G$ in $L^2$ by
\[
A^\theta_G f = P_\theta Gf \qquad \hbox{for all} \quad f \in K_\theta \cap H^\infty.
\]

TTO have attracted a great deal of interest (see, in particular, the book \cite{GMR16}).
ATTO were introduced in
\cite{CP17} and studied, for the first time, by means of EAE. The
motivation of this came from the study of the kernels of ATTO. We have
\beq\label{eq:27}
A^{\theta,\alpha}_G \phi_\theta =0 \iff P_\alpha G \phi_\theta= 0 \iff G\phi_\theta \in K_\alpha^\perp = H^2_- \oplus \alpha H^2_+
\eeq 
with $\phi_\theta \in K_\theta$. This is equivalent to
\begin{align}\label{eq:28}
\left\{
\begin{aligned}
\ol\theta \phi_{1+} &= \phi_{1-}  \\
G \phi_\theta &= \phi_{2-} - \alpha \phi_{2+} 
\end{aligned} \right.
& \implies
\begin{pmatrix}
\ol\theta & 0 \\ G & \alpha \end{pmatrix}
\begin{pmatrix} \phi_{1+} \\ \phi_{2+} \end{pmatrix} =
 \begin{pmatrix}
 \phi_{1-} \\ \phi_{2-} \end{pmatrix} \nonumber \\
 & \iff 
\begin{pmatrix} \phi_{1+} \\ \phi_{2+} \end{pmatrix} \in \ker T_\G,
\end{align}
where $\phi_{1 \pm},\phi_{2\pm} \in H^2_{\pm}$ (indeed $\phi_{1+}=\phi_\theta$) and
\beq\label{eq:29}
\G = \begin{pmatrix}
\ol\theta & 0 \\ G & \alpha \end{pmatrix}.
\eeq
We see that
\beq\label{eq:12ajun23}
\ker A^{\theta,\alpha}_G = P_1 \ker T_\G,
\eeq
where $P_1(x,y)=x$, defines an isomorphism between the two kernels.

As in the previous example, by reformulating the problem of characterising the kernel in terms
of a Riemann--Hilbert problem we obtain a natural candidate for equivalence after extension
with $A^{\theta,\alpha}_G$. Indeed, we have the following:

\begin{thm}\label{thm:1} \cite{CP16,CP17}
\begin{align}
A_G^{\theta,\alpha} & \simast T_\G, & \G&=\begin{pmatrix} \ol\theta & 0 \\ G & \alpha \end{pmatrix}, \qquad \hbox{and so} \nonumber\\
A_G^{\theta} & \simast T_{\G_1}, & {\G_1}&=\begin{pmatrix} \ol\theta & 0 \\ G & \theta \end{pmatrix}.
\end{align}
\end{thm}

We omit the full details of the proof, but we remark that it proceeds in two simpler steps. First,
we introduce the intermediate operator 
\[
P_\alpha g P_\theta + Q_\theta: H^2_+ \to K_\alpha \oplus \theta H^2_+,
\]
where $P_\theta$ and $Q_\theta$ are the orthogonal projections from $H^2_+$ ono $K_\theta$ and $\theta H^2_+$
respectively. Then the proof proceeds by showing that
$A_G^{\theta,\alpha}   \simast P_\alpha g P_\theta + Q_\theta$ and  
$P_\alpha g P_\theta + Q_\theta \simast T_\G$.\\

\begin{rem}{\rm
The EAE of Theorem~\ref{thm:1} is useful in the study of ATTO
because there are tools to study various properties of block Toeplitz operators -- such as invertibility and Fredholmness -- which one can use to study ATTO through EAE, establishing connections that might otherwise be difficult to see.

These tools include relations with the Corona Theorem. Indeed, corona problems, seen as left-invertibility problems, have a strong connection with the invertibility and Fredholmness of block Toeplitz operators
(see, for instance, \cite{CP17} and references therein), as we explain next. In this way, by using Theorem~\ref{thm:1}, the connections of ATTO with the Corona Theorem appear in  a
very natural way.
}
\end{rem}

Let $\CP^\pm$ denote the set of corona pairs in $\DD$ and its complement; i.e.,
\begin{align*}
\CP^+ &= \{H_+=(h_{1+},h_{2+}) \in (H^\infty)^2: \inf_{z \in \DD} (|h_{1+}(z)|+|h_{2+}(z)|)>0\},\\
\CP^- &= \ol{\CP^+}.
\end{align*}
By the Corona Theorem, $H_+ \in \CP^+$ if and only if there exist 
$\wt H_+ =(\wt h_{1+} ,\wt h_{2+} ) \in (H^\infty)^2$ such that
\[
\wt H_+^T H_+ = 1,
\]
and analogously for $H_- \in \CP^-$. Now define $\CP^\pm_M$ as the set of all pairs $(f^M_{1\pm},f^M_{2\pm})$ with $f^M_{i\pm}=r_i f_{i\pm}$, where
$r_i^{\pm 1} \in \GR$, the space of rational $L^\infty$ functions, and $(f_{1\pm},f_{2\pm}) \in \CP^\pm$.

We have the following \cite{CDR10}.

\begin{thm}
Let $G \in (L^\infty)^{2 \times 2}$ and suppose that there exist
$\phi^M_+, \phi^M_- \in \CP^\pm_M$ such that
\[
G\phi^M_+=\phi^M_-.
\]
Let, moreover, $\gamma=\det G$. Then $T_G$ is Fredholm if and only if $T_\gamma$ is Fredholm and, in that case,
the Fredholm indices of the two operators are the same; i.e.,
\[
\ind T_G = \ind T_\gamma.
\]
\end{thm}
\begin{thm}\cite{CDR10}
Let $\gamma \in L^\infty$. Then $T_\gamma$ is Fredholm if and only if $\gamma\in \G L^\infty$ and $\gamma$ admits a factorization
\beq\label{eq:A6jul30}
\gamma=\gamma_- z^k \gamma_+ , \qquad \hbox{with} \quad
z \in \ZZ, \gamma_+^{\pm 1} \in H^2_+, \ol\gamma_-^{\pm 1} \in H^2_+
\eeq
(called a {\em Wiener--Hopf factorization in $L^2$}); we then have that $\ind T_\gamma=-k$.
\end{thm}
\begin{thm}\label{thm:A3jul30} \cite{CDR10}
If $G \in (L^\infty)^{2 \times 2}$ with $\det G=\gamma$, satisfying \eqref{eq:A6jul30}
and if there exist $\phi_\pm \in \CP^\pm$ such that
\beq\label{eq:A7jul30}
G\phi_+=\phi_-,
\eeq
then $T_G$ is invertible, injective, or surjective if and only if $T_\gamma$ is invertible $(k=0)$,
injective $(k \ge 0)$, or surjective $(k \le 0)$, respectively.
\end{thm}
Moreover, we have:
\begin{thm}\label{thm:A4jul30}
Let $G \in (L^\infty)^{2 \times 2}$ with $\det G=\gamma$, satisfying \eqref{eq:A6jul30} with $k=0$ (called a canonical Wiener--Hopf factorization in $L^2$) and let $\phi_+ \in (H^\infty)^2$ and $\phi_- \in (\ol{H^\infty})^2$ satisfy \eqref{eq:A7jul30}. If $\phi_+ \in \CP^+$
(respectively, $\phi_- \in \CP^-$) then $T_G$ is invertible if and only if $\phi_- \in \CP^-$
(respectively,  $\phi_+ \in \CP^+$).
\end{thm}
These results can be applied to a Toeplitz operator with matrix symbol of the form
\beq\label{eq:A8jul30}
G  = \begin{pmatrix}
\ol\theta & 0 \\ g & \theta
\end{pmatrix},
\eeq
which is EAE to the TTO $A^\theta_g$, taking into account that $\gamma=\det G=1$ satisfies 
\eqref{eq:A6jul30} with $k=0$. We present here two examples, which will illustrate this and which
will be related with dual truncated Toeplitz operators, defined on the orthogonal complement of a model space, in Section~\ref{sec:4jul30}.

\begin{prop}\label{prop:A5aug3}
Let $G$ be defined by \eqref{eq:A8jul30}, with 
\beq\label{eq:A9jul30}
g^{-1} \in H^\infty  \quad \hbox{and} \quad \ol\theta g^{-1} \in \ol{H^\infty}.
\eeq
Then $T_G$ and $A^\theta_g$ are invertible if and only if $g \in \G H^\infty$.
\end{prop}
\beginpf
We have that
\[
\begin{pmatrix}
\ol\theta & 0 \\ g & \theta
\end{pmatrix}
\begin{pmatrix}
g^{-1} \\ 0 
\end{pmatrix} =\begin{pmatrix}
\ol\theta g^{-1} \\ 1 
\end{pmatrix},
\]
where $(\ol\theta g^{-1},1) \in \CP^-$, so, by Theorems~\ref{thm:A3jul30}
and \ref{thm:A4jul30}, $T_G$ is invertibke if and only if $(g,0) \in \CP^+$, which is
equivalent to $g \in \G H^\infty$.
\end{proof}

The assumption \eqref{eq:A9jul30} is satisfied, in particular, if
$g^{-1} \in K_{z\theta} \cap L^\infty$.\\

Using the previous results relating the invertibility and Fredholmness of $2 \times 2$ block Toeplitz operators with the solutions of
certain corona problems and the EAE of TTO with Toeplitz operators with symbols
of the form \eqref{eq:A8jul30}, one also obtains the following results for analytic symbols $g$ 
(or anti-analytic symbols, in which case it is enough to consider the adjoint operator).

\begin{thm}\label{thm:A6aug3}
Let $g \in H^\infty$ and denote by $g_i$ the inner factor of $g$. Then:
\\
(i) $A^\theta_g$ is Fredholm if and only if $\gamma=\GCD(\theta,g_i)$ is
a finite Blaschke product and $\ol\gamma(\theta,g) \in \CP^+$;
\\
(ii) $A^\theta_g$ is invertible if and only if $(\theta,g) \in \CP^+$;
\\
(iii) $\ker A^\theta_g = \frac{\theta}{\gamma} K_\gamma$.
\end{thm}

By studying the invertibility, Fredholmness, and kernel of block Toeplitz operators with symbols of the form \eqref{eq:A8jul30}, where $g$ is replaced by $g-\lambda$ for $\lambda\in \CC$,
one can also study the spectrum of the TTO $A^\theta_g$. This was done in
\cite{CP16}.\\

%
%

Furthermore, although the ATTO $A_G^{\theta,\alpha}$ can be the zero operator while $T_\G$
is never zero, one can reformulate the question of when $A_G^{\theta,\alpha}$ is zero
by asking when its kernel is equal to $K_\theta$ and using \eqref{eq:12ajun23}.

Indeed, this was answered by Jurasik and {\L}anucha~\cite{JL16},
the result for symmetric TTO having been given by 
 Sarason~\cite[Thm.~3.1]{sarason07}. 
We shall give an alternative and simpler proof, based on the ideas above.
\begin{thm}
If $G \in L^2$ then $A^{\theta,\alpha}_G=0$ if and only if $G \in \alpha H^2_+ \oplus \ol\theta \ol{H^2_+}$.
\end{thm}
\beginpf
Clearly, if $\phi \in K_\theta$ and $G = \alpha h_1+  \ol\theta \ol{h_2}$
for $h_1,h_2 \in H^2_+$, then
we have $ G \phi = \alpha( h_1 \phi) + \ol{h_2}( \ol\theta \phi) \in \alpha H^2_+ + H^2_-$,
so that $A^{\theta,\alpha}_G \phi=0$.

Conversely, $A^{\theta,\alpha}_G=0$ if and only if $\ker A^{\theta,\alpha}_G = K_\theta$; that is,
$P_1 \ker T_\G = K_\theta$, 
which is equivalent to
\[
\forall \phi_{1+} \in K_\theta \quad \exists \phi_{1-},\phi_{2-} \in H^2_-, \phi_{2+} \in H^2_+ :
\quad \begin{pmatrix} \ol\theta & 0 \\ G & \alpha \end{pmatrix} \begin{pmatrix}\phi_{1+} \\ \phi_{2_+} \end{pmatrix} =  \begin{pmatrix}\phi_{1-} \\ \phi_{2_-} \end{pmatrix},
\]
i.e.,
\[
\forall \phi_{1+} \in K_\theta \quad \exists \phi_{2 \pm} \in H^2_{\pm}: \quad G \phi_{1+} + \alpha \phi_{2_+} = \phi_{2-}. 
\]
We now use the test functions $\phi_{1+}=S^* \theta = \frac{\theta -\theta(0)}{z}$
and $\phi_{1+} = k^\theta_0 = 1-\ol{\theta(0)}\theta$ (the reproducing kernel at $0$). Thus
there exist $\phi_{2\pm}$ and $\psi_{2\pm} \in H^2_\pm$ such that
\begin{align}
G(\theta-\theta(0))+ z \alpha \phi_{2+}  &= z\phi_{2-},\label{eq:a35}
\\
G(1-\ol{\theta(0)}\theta) + \alpha \psi_{2+} &= \psi_{2-}.\label{eq:a36}
\end{align}
Taking the linear combination
$\ol{\theta(0)}\times$\eqref{eq:a35}$+$\eqref{eq:a36},
we obtain
\begin{align}
G(1-\theta(0)\ol{\theta(0)}) & = -\ol{\theta(0)}z\alpha \phi_{2+}+\ol{\theta(0)}z\phi_{2-} -\alpha \psi_{2+}+\psi_{2-}\label{eq:a37}.
\end{align}
On applying $P^+$ to  \eqref{eq:a37} 
we obtain
\begin{align}
(1-\theta(0)\ol{\theta(0)}) P^+ G &= -\alpha(z  \ol{\theta(0)}\phi_{2+}+\psi_{2+})+ \ol{\theta(0)}P^+(z \phi_{2-}), \label{eq:a39}
\end{align}
That is, $P^+ G \in \alpha H^2_+ + \CC$. Now the adjoint of $A^{\theta,\alpha}_G$ is
$A^{\alpha,\theta}_{\ol G}$ and it is also the zero operator, and we conclude that
$P^+ \ol G \in \theta H^2_+ + \CC$, or equivalently $P^- G \in \ol \theta \ol{H^2_+}+\CC$.

We may therefore write 
\[
G= \alpha g_{1+}+ \ol\theta \ol{g_{2+}} + k
\]
with $g_{1+}, g_{2+} \in H^2_+$ and $k \in \CC$. The first two terms give zero ATTO, as we
have already seen. If $A^{\theta,\alpha}_k=0$ and $k\ne 0$ then $K_\theta \subseteq \alpha H^2_+$: however, this is not possible, since $K_\theta$ is a Toeplitz kernel.
Thus $k=0$ and the result follows.
\endpf

From the proof above we have the surprising consequence.

\begin{cor}
If $\ker A^\theta_G$ contains both the functions $k^\theta_0=1-\ol{\theta(0)}\theta$ and $\wt k^\theta_0= S^*\theta$, then $\ker A^\theta_G= K_\theta$.
\end{cor}

A similar result holds for $A^{\theta,\alpha}_G$, expressed using two test functions for the operator and two for its adjoint.

\begin{rem}{\rm
From the relation \eqref{eq:12ajun23} one sees that, although $\ker T_\G$ consists of vector functions, it
must be isomorphic to a space of scalar functions, and thus exhibits a {\em scalar-type behaviour}.
This was studied in \cite{CP20}. }
\end{rem}

Finally, the equivalence after extension of Theorem~\ref{thm:1} allows for
a better understanding of the different behaviour of TTO when 
compared with Toeplitz operators highlighted in several works: TTO behave,
in a certain way, as block Toeplitz operators, whose properties are, in general, 
very different from those of Toeplitz operators with scalar symbols.

The techniques used here can be used to give an alternative proof of a theorem of O'Loughlin
regarding near invariance of kernels of TTO.

\begin{thm} \cite[Thm~4.4]{ryan20}, \cite[Prop.~5.7]{CKP24}.
If $\ker A^\theta_g$ contains a function that does not vanish at 0, then it is nearly
$S^*$-invariant.
\end{thm}

\beginpf
We have $\phi_{1+} \in \ker A^\theta_g$ with
\begin{align}
\ol\theta \phi_{1+} &= \phi_{1-} ,\label{eq:may3} \\
g\phi_{1+} + \theta \phi_{2+} & = \phi_{2-}, \label{eq:may1}
\end{align}
where $\phi_{1\pm},\phi_{2\pm} \in H^2_\pm$   and $\phi_{1+}(0) \ne 0$.
Suppose that another function $\psi_{1+}  \in \ker A^\theta_g$,  so that it satisfies
\begin{align}
\ol\theta \psi_{1+} &= \psi_{1-}, \label{eq:may4} \\
g\psi_{1+} + \theta \psi_{2+} & = \psi_{2-}, \label{eq:may2}
\end{align}
where $\psi_{1\pm},\psi_{2\pm} \in H^2_\pm$,
and suppose that now $\psi_{1+}(0)=0$. 

From \eqref{eq:may1} and \eqref{eq:may2} we have
\[
\theta \phi_{2+}\psi_{1+}- \theta\psi_{2+}\phi_{1+} = \phi_{2-}\psi_{1+}-\psi_{2-}\phi_{1+},
\]
and on multiplying by $\ol\theta$ this rearranges to give
\[
-\phi_{2-}\psi_{1+}\ol\theta + \phi_{2+}\psi_{1+} = -\psi_{2-}\phi_{1+}\ol\theta + \psi_{2+}\phi_{1+} .
\]
In view of \eqref{eq:may3} and \eqref{eq:may4} we have
\[
 \phi_{2+}\psi_{1+} - \psi_{2+}\phi_{1+}= \phi_{2-}\psi_{1-}- \psi_{2-}\phi_{1-} =0
\]
since the left-hand side is in $H^1_+$ and the right-hand side in $H^1_-$.
Since $\psi_{1+}(0)=0$ and $\phi_{1+}(0) \ne 0$ we have that $\psi_{2+}(0)=0$
and so $\ol z \psi_{2+} \in 
 H^2_+$. From \eqref{eq:may2} we now have
\[
g(\ol z\psi_{1+}) + \theta (\ol z\psi_{2+})  = \ol z\psi_{2-} \in H^2_-,
\]
showing that $\ol z\psi_{1+} \in \ker A^\theta_g$.
\endpf

Generalizations of this result are possible, asserting the near $\ol\alpha$-invariance
of $\ker A^\theta_g$,
for an inner function $\alpha$, provided that there is a function $\phi_{1+} \in \ker A^\theta_g$ that has no common inner factor with $\alpha$. The proof is very similar.

\begin{rem}\label{rem:3.6}{\rm
O'Loughlin~\cite{ryan22} has proved a vectorial generalisation of
Theorem~\ref{thm:1}. Here we work with $n \times n$ matricial inner functions,
which are defined to be elements of $(H^\infty)^{n \times n}$ such that
$\Theta(z)$ is unitary for almost all $z \in \TT$. From these we may construct the model space
$K_\Theta = (H^2_+)^{n } \cap \Theta (H^2_-)^{n}$.

For $G \in (L^\infty)^{n \times n}$ we   define the TTO $A^\Theta_G$ acting on
$K_\Theta$ by
\[
A^\Theta_G f = P_\Theta (Gf),
\]
where $P_\Theta$ denotes the orthogonal projection from $ (H^2_+)^{n }$ onto $K_\Theta$.
It can then be shown that $A^\Theta_G$ is equivalent after extension to
the block Toeplitz operator $T_\G$ with symbol 
\[
\G = \begin{pmatrix}
\Theta^* & 0 \\ G & \Theta
\end{pmatrix}.
\]
Further generalisations, including to the non-Hilbert case, are given in \cite{ryan22}.
}
\end{rem}

\begin{rem}{\rm 
It is possible to extend the definition of  truncated Toeplitz operators so that
they act on a wider class of spaces than model spaces.

Hartmann and Ross~\cite{HR13} showed that for $G \in L^\infty$ and $M=h_+K_\theta$
a nearly invariant subspace, with $K_\theta$ a model space and $h_+$ an isometric multiplier
from $K_\theta$ onto $M$, 
the
truncated Toeplitz operator $A^M_G$ defined by
\[
A^M_G f = P_M (Gf),
\]
with $P_M$ the orthogonal projection from $L^2$ onto $M$, is unitarily equivalent to
the ``standard'' truncated Toeplitz operator $A^\theta_{|h_+|^2 G}$ acting on $K_\theta$.
It follows from Theorem~\ref{thm:1} that we have
\[
A^M_G \simast T_\G, \quad \hbox{where now} \quad
\G= \begin{pmatrix} \ol\theta & 0 \\ |h_+|^2G & \theta \end{pmatrix}.
\]
}
\end{rem}

\section{Asymmetric dual truncated Toeplitz operators}
\label{sec:4jul30}

For $\theta$ an inner function let $K_\theta^\perp= H^2_- \oplus \theta H^2_+=L^2 \ominus K_\theta$ and let
\[
Q_\theta = I_{L^2}-P_\theta = P^- + \theta P^+ \ol\theta I_{L^2}
\]
be the orthogonal projection from $L^2$ onto $K^\perp_\theta$. We define the 
{\em asymmetric dual truncated Toeplitz operator (ADTTO)\/} with symbol $G$
by
\[
D_G^{\theta,\alpha}: K^\perp_\theta \to K^\perp_\alpha, \qquad D^{\theta,\alpha}_G= Q_\alpha G {Q_\theta}_{|K_\theta^\perp}
\]
and, if $\theta=\alpha$,
\[
D_G^\theta: K_\theta^\perp \to K_\theta^\perp, \qquad D_G^\theta=Q_\theta G{ Q_\theta}_{|K_\theta^\perp}.
\]
These operators were introduced in \cite{ding} and further studied in \cite{CKLP}.
It was clear from the beginning that DTTO behave very differently from TTO. We have, for instance,
\begin{itemize}
\item $D_G^{\theta,\alpha}$ is bounded if and only if $G \in L^\infty$.
\item It is compact if and only if $G=0$.
\item If it is Fredholm, then $G \in \G L^\infty$.
\end{itemize}
None of these properties holds for TTO.
However, we shall show, using the concept of EAE, that
the two types of operator are related and, 
under  certain conditions, one may be studied from the other. 

DTTO are a very natural class of operators to study, and they possess unexpected properties and lead to new questions. For example, the study of the invariant subspaces of the dual truncated shift shows that they have
a much richer structure than those for the truncated shift on model spaces~\cite{CR21};
this raised the question of identifying the invariant subspaces of $S \oplus S^*$, which was solved by Timotin~\cite{timotin}.

The motivation to study a DTTO using its EAE with another operator came also from the study of 
its kernel in terms of a Riemann--Hilbert problem.
In order to do this, we start by noting that
\[
f_{\theta^\perp} \in K_\theta^\perp \iff f_{\theta^\perp}= f_- + \theta \til f_+, \quad \hbox{with}
\quad f_- \in H^2_- \hbox{ and } \til f_+ \in H^2_+,
\]
so
\begin{align}
D^{\theta,\alpha}_G f_{\theta^\perp}=0 &\iff Q_\alpha G  f_{\theta^\perp} = 0
\iff G(f_- + \theta \til f_+) \in K_\alpha \nonumber \\
& \iff 
\left \{ 
\begin{aligned}
G(f_- + \theta \til f_+)&= \psi_+ \\
\ol\alpha G (f_-+ \theta \til f_+) &= \psi_-
\end{aligned}
\right.
\qquad \hbox{with} \quad \psi_\pm \in H^2_\pm. \label{eq:37}
\end{align}
We write the last equations in matrix form as
\beq\label{eq:38}
\left( \begin{pmatrix} G\theta & -1 \\ G\theta\ol\alpha & 0\end{pmatrix} P^+
+ \begin{pmatrix} G & 0 \\ \ol\alpha G & -1 \end{pmatrix} P^- \right)
\begin{pmatrix} f_1 \\ f_2 \end{pmatrix} = \begin{pmatrix}0 \\ 0 \end{pmatrix}
\eeq
with 
\[ 
f_1 = f_- + \til f_+ \qquad \hbox{and} \qquad f_2= \psi_- + \psi_+.
\]
Thus \eqref{eq:38} is equivalent to
\beq \label{eq:40}
(AP^+ + BP^-)f=0,
\eeq
i.e., $f \in \ker (AP^++BP^-)$, where $AP^++BP^-$
is a paired operator with
\beq\label{eq:25A3aug}
A= \begin{pmatrix} G\theta & -1 \\ G\theta\ol\alpha & 0\end{pmatrix} \qquad \hbox{and}
\qquad B= \begin{pmatrix} G & 0 \\ \ol\alpha G & -1 \end{pmatrix}.
\eeq
From \eqref{eq:37}--\eqref{eq:40} we also get a simple isomorphism between
$\ker D^{\theta,\alpha}_G$ and $\ker(AP^++BP^-)$, namely
\beq\label{eq:25ajun23}
\ker D^{\theta,\alpha}_G = (P^-+\theta P^+) P_1 \ker (AP^++BP^-).
\eeq
The operator $AP^++BP^-$ is thus a natural candidate for an operator that is EAE to $D^{\theta,\alpha}_G$.
Indeed we have:

\begin{thm} 
\label{thm:4.1aug3}
\cite{CKLP}
$ D_G^{\theta,\alpha} \simast AP^++BP^-$.
\end{thm}


As an immediate consequence of Theorem~\ref{thm:4.1aug3} and the properties of paired operators we have the following, with the notation of Theorem~\ref{thm:4.1aug3}.

\begin{cor}\label{cor:4.1A3aug}
The operator $D^{\theta,\alpha}_G$ is semi-Fredholm (resp., Fredholm, invertible)
if and only if $AP^+ + BP^-$ is semi-Fredholm (resp., Fredholm, invertible)
on $(L^2)^2$. If $D_G^{\theta,\alpha}$ is semi-Fredholm then $G \in \G L^\infty$.
\end{cor}
\beginpf
The result follows from Theorem~\ref{thm:4.1aug3} and the fact that a 
necessary condition for the operator $AP^+ + BP^-$ to be semi-Fredholm
is that 
\[
\essinf_{t \in \TT} |\det A(t)|>0 \qquad \hbox{and} \qquad \essinf_{t \in \TT} |B(t)|>0,
\]
(\cite[Chap. 5]{MP})
and, in our case, we have $|\det A|=|\det B|=|G|$.
\endpf

In particular, for $\alpha=\theta$, it follows that, denoting by $\sigma(A)$ and $\sigma_e(A)$
respectively  the 
spectrum and essential spectrum of an operator $A$, we have
\[
G(\TT) \subseteq \sigma_e(D^\theta_G) \subseteq D^\theta_G.
\]
In view of Corollary~\ref{cor:4.1A3aug}, we see that the case when $G \in \G L^\infty$
is particularly important. In that case, we see that $A,B$ defined by \eqref{eq:25A3aug},
are invertible in $(L^\infty)^{2 \times 2}$ and
\[
D_G^{\theta,\alpha} \simast AP^+ + BP^- = B(P^+B^{-1}AP^++P^-)(I+P^{-}B^{-1}AP^+),
\]
where $B$ and $I+P^-B^{-1}AP^+$ are invertible in $(L^2)^2$, with
\[
(I+P^-B^{-1}AP^+)^{-1}=I-P^-B^{-1}AP^+.
\]
 So, writing
 \[
 C=B^{-1}A = \begin{pmatrix} \theta & -G^{-1} \\ 0 & -\ol\alpha \end{pmatrix},
 \]
 we have that 
 \begin{align*}
 D_G^{\theta,\alpha} & \simast P^+CP^++P^- = (CP^++P^-)(I-P^-CP^+) \sim  CP^+ +P^- \\
 & \simast T_C \sim T_\G \simast A^{\alpha,\theta}_{G^{-1}}
 \end{align*}
 with
 \[
 \G= \begin{pmatrix} 0 & 1 \\  1 & 0 \end{pmatrix}
 C \begin{pmatrix} 0 & 1 \\ -1 & 0 \end{pmatrix}
 = \begin{pmatrix} \ol\alpha & 0 \\ G^{-1} & \theta \end{pmatrix},
 \]
 where we took into account the invertibility of $I-P^-CP^+$, the EAE between a paired operator and a Toeplitz operator (Theorem~\ref{thm:2.1aug3}), and the EAE between a TTO and a $2 \times 2$
 block Toeplitz operator (Theorem~\ref{thm:1}).
 
 We have the following result, which is rather surprising in view of the different properties
 of the two types of operators involved.
 
 \begin{thm}\cite{CKLP}\label{thm:4.1Baug3}
 If $G \in \G L^\infty$, then $D^{\theta,\alpha}_G \simast A^{\alpha,\theta}_{G^{-1}}$.
 \end{thm}
 
 \begin{cor} \cite{CKLP}
  If $G \in \G L^\infty$, then $D^{\theta}_G \simast A^{\theta}_{G^{-1}}$.
  \end{cor}
  
  As a consequence, the invertibility and Fredholmness properties of DTTO and, therefore, their spectrum, can   be studied from those of TTO when the symbol is invertible in $L^\infty$.
For example, from Proposition~\ref{prop:A5aug3} and Theorem~\ref{thm:A6aug3} for TTO we have the following:
  
  \begin{thm}\cite{CKLP}
  If $G \in H^\infty$ and $\ol\theta G \in \ol{H^\infty}$, then
  $\sigma(D^\theta_G)=\clos G(\DD)$.
  \end{thm}
  
  \begin{thm}\cite{CKLP}
  Let $G \in \G L^\infty$. Suppose that $G^{-1} \in H^\infty$, and let $G^{-1}=\beta O_G$ be its inner--outer factorization, with $\beta$ inner and $O_G$ outer. Write $\gamma= \gcd(\theta,\beta)$. Then\\
  (i) $D^\theta_G$ is Fredholm if and only if $\gamma$ is a finite Blaschke product.;
  \\
  (ii) $D^\theta_G$ is invertible if and only if $(\theta,\beta) \in \CP^+$.
  \end{thm}
  
  \begin{cor}
  If $G \in \G H^\infty$, then $D^\theta_G$ is invertible.
  \end{cor}
  
  \begin{cor}
  If $\theta$ is a singular inner function, then $D^\theta_G$ is Fredholm if and only if it is invertible.
  \end{cor}
  
  These results establish a connection of DTTO with corona problems, by using the EAE of Theorem~\ref{thm:4.1Baug3} when $G \in \G L^\infty$ and the relations of TTO with block Toeplitz operators and corona problems. The connections of $2\times 2$ block
  Toeplitz operators with corona problems can be extended to paired operators with $2 \times 2$ coefficients $A$ and $B$ as follows.
  
  \begin{thm}\label{thm:4.1Haug3}
  \cite[Thm.~8.1]{CKLP}. Let $A,B \in (L^\infty)^{2 \times 2}$ be such that
  $\det A=\phi f_+$ and $\det B=\phi f_-$, where $\phi \in L^\infty$, $\phi(t) \ne 0$ a.e., and
  $f_+, \ol{f_-} \in H^\infty \setminus \{0\}$. If there exist $H_\pm \in \CP^\pm$ satisfying
  \[
  AH_+ + B H_- =0
  \]
  with $AH_+(t) \ne 0$ a.e.\ on $\TT$, then the operator $AP^++BP^-$
  is injective in $(L^2)^2$.
  \end{thm}
  Taking Theorem~\ref{thm:4.1aug3} into account, we get the following:
  
  \begin{thm}\cite[Thm.~8.3]{CKLP} \label{thm:4.1Iaug3}
  Let $G \in L^\infty$ with $G(t)\ne 0$ a.e., and
  let $\theta,\alpha$ be inner functions with $\alpha \le \theta$ (i.e., $\alpha$ divides $\theta$).
  If there exist $h_{1+},h_{2+},\ol{h_{2-}} \in H^\infty$ such that
  $\ol\alpha h_{1+}=: h_{1-} \in \ol{H^\infty}$, and
  $(h_{1+},h_{2+}) \in \CP^+$, $(h_{1-},h_{2-}) \in \CP^-$, with
  \beq\label{eq:2.6Haug3}
  G(h_{2-}+\theta h_{2+})=h_{1+},
  \eeq
  then $\ker D_G^{\theta,\alpha}=\{0\}$.
  \end{thm}
  \beginpf
  We have $\det A=\theta\ol\alpha$, $\det B=-G$ for $A$, $B$, defined by 
  \eqref{eq:25A3aug}.  From \eqref{eq:2.6Haug3} we get
  \begin{align*}
  \begin{cases}
    G (\theta h_{2+} + h_{2-}) - h_{1+}&=0,\\
    \ol\alpha G \theta h_{2+} + \ol\alpha Gh_{2-} -h_{1-}&=0,
  \end{cases}
\end{align*}
that is, 
\[
A \begin{pmatrix} h_{2+} \\ h_{1+} \end{pmatrix} + B \begin{pmatrix} h_{2-} \\ h_{1-}
\end{pmatrix} =0,
\]
so, by Theorem~\ref{thm:4.1Haug3}, $AP^+ + BP^-$ is injective in $(L^2)^2$ and therefore
$D^{\theta,\alpha}_G$ is injective in $K^\perp_\theta$.  
  \endpf
  
  As an example, take $D^{\theta,\alpha}_G$ with $\alpha \le \theta$ and $G=\dfrac{h_{1+}}{c_1+\theta c_2}$, where
  $h_{1+} \in K_\alpha^\infty \setminus \{0\}$ and $c_1,c_2 \in \CC\setminus\{0\}$ with
  $|c_1| \ne |c_2|$. Then,
  for $h_{1-}=\ol\alpha h_{1+}$, $h_{2_+}=c_2$ and $h_{2-}=c_1$, the assumptions of Theorem~\ref{thm:4.1Iaug3} are satisfied and we deduce that
  $\ker D_G^{\theta,\alpha}=\{0\}$.\\


As we have seen for Toeplitz and truncated Toeplitz operators one can show from 
\eqref{eq:38} and \eqref{eq:25ajun23} that
 the fact that a given
function lies in the kernel of a DTTO gives further information on the kernel. 
This is related with the question of characterizing the symbols of DTTOs that
correspond to the zero operator.
We give some examples,
considering only the case $\alpha=\theta$.

\begin{thm}
(i) $\theta\in\ker D^\theta_G \iff G \in \ol \theta K_\theta \cap L^\infty \subset\ol z \ol{H^\infty}$.\\
(ii) If (i) holds and $G \ne 0$ then $\ker D^\theta_G = \theta K_{zI}$, where $\ol G= zIO$ is the inner-outer
factorization, with $I$ inner and $O$ outer.
\end{thm}

\beginpf
(i) From \eqref{eq:37} we have:
$\theta\in\ker D^\theta_g \iff G\theta \in K_\theta$; that is,
\begin{align*}
\left\{
\begin{aligned}
G\theta-\psi_+ &=0, \\
G- \psi_- & =0,
\end{aligned}
\right.
\end{align*}
with $\psi_\pm \in H^2_\pm$. That is, $G=\psi_-$ with
$\ol\theta \psi_+=\psi_-$. Equivalently, $G \in \ol\theta K_\theta \cap L^\infty \subset \ol z \ol{H^\infty}$.

(ii) 
We have  that $f=f_-+\theta f_+$ (with the usual notation) lies in $\ker D^\theta_G$  if and only if $Gf \in K_\theta$, i.e.,
\begin{align*} 
& \left\{
\begin{aligned}
G\theta f_+ -\psi_+  + Gf_- &=0,  \\
Gf_+ +  \ol\theta Gf_- - \psi_- &=0,
\end{aligned}
\right.\\
\iff
& \left\{ 
\begin{aligned}
G\theta f_+ -\psi_+   &=Gf_- \implies f_-=0, \\
Gf_+ &= \psi_- -   \ol\theta Gf_-  \in H^2_-   ,
\end{aligned}
\right.
\iff
&\left\{
\begin{aligned}
 f_- &=0, \\
G\theta f_+ &= \psi_+ \\
Gf_+ &= \psi_-
\end{aligned}
\right.
\end{align*}

That is, $G\theta f_+ \in H^2_+$ and $Gf_+ \in H^2_-$, so 
$\ol z \ol I f_+ \in H^2_-$, where $\ol G=zIO$
is the inner-outer factorization. Hence  
$f \in \theta K_{zI}$.

Conversely, if $f \in \theta  K_{zI}$, say $f=\theta f_+$ with $f_+ \in K_{zI}$, then for $G \in \ol\theta K_\theta$, we have $G f \in H^2_+$ and $\ol \theta G f \in H^2_-$, so $f \in \ker D^\theta_G$.
\endpf

We can derive another similar result by the same method.

\begin{thm}
(i) $\ol z\in\ker D^\theta_G \iff 
G \in z K_\theta \cap L^\infty \subset z H^\infty.$\\
(ii)  If (i) holds and $G \ne 0$ then 
$ \ker D^\theta_G = \ol z \ol I K_{zI} = \ol z \ol{K_{zI}}$, where $G=zIO$ is 
the inner-outer factorization with $I$ inner and $O$ outer.
\end{thm}

\beginpf
(i)   $\ol z \in D^\theta_G \iff G\ol z  \in K_\theta$; that is,
\begin{align*}
\left\{
\begin{aligned}
G\ol z &=\psi_+, \\
\ol\theta G \ol z &= \psi_- ,
\end{aligned}
\right. \iff
\left\{
\begin{aligned}
G &=z\psi_+, \\
\ol\theta \psi_+ &= \psi_- ,
\end{aligned}
\right.
\end{align*}
with $\psi_\pm \in H^2_\pm$. Thus,
$G \in z K_\theta \cap L^\infty$.
The converse is clear.

(ii) Next, $f=f_- + \theta f_+$ (with the usual notation) lies in 
$\ker D^\theta_G$ if and only if
we have $Gf \in K_\theta$, i.e.,
\begin{align*}
\left\{
\begin{aligned}
Gf_- &=\phi_+, \\
\ol\theta G f_-+Gf_+ &= \phi_- ,
\end{aligned}
\right. \iff
\left\{
\begin{aligned}
Gf_- &=\phi_+, \\
\ol z\ol\theta G f_-+\ol z Gf_+ &= \ol z \phi_- ,
\end{aligned}
\right.
\end{align*}
with $\phi_\pm \in H^2_\pm$.
then, since $\ol z\ol\theta G \in \ol {H^\infty}$, we have $\ol z\ol \theta G f_- \in H^2_-$ and so $\ol z G  f_+ =0$ and $f=f_-$.  
Equivalently,
\begin{align*}
\left\{
\begin{aligned}
Gf_- &=\phi_+, \\
f_+ &= 0, \\
\theta \ol G \ol z \ol f_- &= \ol z \phi_- ,
\end{aligned}
\right.
\end{align*}

Thus $G f_- \in H^2_+$ and $\ol \theta G f_- \in H^2_-$. Let $G=zIO$ be an inner-outer factorization.
We have  
$z I  f_-\in H^2_+$ implying that $\ol z \ol I (\ol z \ol{f_-}) \in H^2_-$;
since   $\ol z \ol{f_-} \in H^2_+$ we have $\ol z \ol{f_-} \in K_{zI}$.

Once again, we can reverse these steps to see that every element of $\ol z \ol{K_{zI}}$ is in 
$\ker D^\theta_G$,
\endpf

An immediate consequence of the two previous theorems is the following, which shows
that there exist sets of functions that cannot be contained in the kernel of any DTTO that
is not the zero operator.

\begin{cor}
If  $\theta, \ol z \in \ker D^\theta_g$, then $D_g=0$.
\end{cor}

From this we get a known result, which is presented here in the light of
the concept of EAE.

\begin{cor}
$D^\theta_G=0 \iff G=0$.
\end{cor}

These results can be best understood with an example. Let $\theta(z)=z^2$ so that
$K_\theta= \spam\{1,z\}$.
With $G(z)=\ol z$ we have $\ker D^\theta_g = \CC z^2$. Whereas, with $G(z)=z^2$
we have $\ker D^\theta_g = \spam \{\ol z,\ol z^2 \}$.

\section{Multiband truncated Toeplitz operators}

As explained in detail in \cite{COP22}, multiband signals occur in speech processing, as an alternative to
the Paley--Wiener spaces of functions with Fourier transforms lying in $L^2(-b,b)$, for $b>0$, when
low as well as high frequencies are to be excluded. More recently, dual band filters have become key 
components in communication devices such as cell phones.

Restricting  to dual band operators, we let $\theta$ be an inner function in the Hardy space $H^\infty(\CC^+)$ of the upper half-plane, and let
$\phi$ and $\psi$ be unimodular functions in $L^\infty(\RR)$ such that
the subspaces $\phi K_\theta$ and $\psi K_\theta$ are orthogonal in $L^2(\RR)$.
Here $K_\theta= H^2(\CC^+) \ominus \theta H^2(\CC^+)$ as we are working in the 
upper halfplane rather than the disc.

The motivating example is $\theta(s)=e^{i(b-a)s}$ with $b>a>0$ and with
$\phi(s)=e^{-ibs}$ and $\psi(s)=e^{ias}$. This can be seen to correspond to the inverse Fourier transform
of $L^2((-b,-a) \cup L^2(a,b))$.

\begin{prop}\cite[Prop.~2.1]{COP22}
With the notation above, $\phi K_\theta$ and $\psi K_\theta$ are orthogonal spaces if and only if $A^\theta_{\ol\phi\psi}=0$.
\end{prop}

Recall that $A^\theta_{\ol\phi\psi}=0 \iff \ol\phi\psi \in \theta H^2 + \ol\theta \ol{H^2}$
by \cite[Thm.~3.1]{sarason07}. Under this condition, we see that 
\[
f^M \in \phi K_\theta \oplus \psi K_\theta \iff f^M = \phi f_{\theta,1} + \psi f_{\theta,2},
\]
where 
\[
 f_{\theta,1} = P_\theta \ol\phi f_M \qquad \hbox{and} \qquad 
 f_{\theta,2} = P_\theta \ol \psi f_M
 \]
 and the orthogonal projection $P_M$ from $L^2$ onto $M=\phi K_\theta \oplus \psi K_\theta$ is given by
 \[
 P_M f = \phi P_\theta \ol \phi f + \psi P_\theta \ol \psi f  \qquad (f \in L^2).
 \]

 If $g \in L^2$, the truncated Toeplitz operator $A^M_g$ is densely defined in $M$ by
 \[
 A_g^M f = P_M gf \qquad (f \in L^\infty \cap M),
 \]
 the density of $L^\infty \cap M$ following from the density
 of $L^\infty \cap K_\theta$ in $K_\theta$. 
 If this operator is bounded, we
 also denote by $A^M_g$ its unique bounded extension to $M$.

It can be shown that $A^M_g$ is unitarily equivalent to the block truncated Toeplitz operator
\[
W = \begin{pmatrix}
A^\theta_g & A^\theta_{\ol\phi \psi g} \\
A^\theta_{\ol\psi \phi g} & A^\theta_g
\end{pmatrix},
\]
and hence it is the zero operator if and only if each of the four truncated Toeplitz operators
composing $W$ are zero.

 Let us study the kernel of $A^M_g$ as we have done in the previous
 sections for Toeplitz operators, ATTOs and DTTOs. We have that, for $f \in M$,
 \[
 A^M_g f =0 \iff \phi P_\theta \ol\phi f + \psi P_\theta \ol\psi g f=0.
 \]
 Since $\phi K_\theta \perp \psi K_\theta$, this is equivalent to
 \[
 P_\theta \ol\phi gf=0 \qquad \hbox{and} \qquad P_\theta \ol\psi gf=0.
 \]
 Equivalently,
\[
 P_\theta \ol\phi g (\phi f_{\theta,1}+ \psi f_{\theta,2}) =0  \qquad\hbox{and} \qquad
 P_\theta \ol\psi g(\phi f_{\theta,1}+ \psi f_{\theta,2}) =0
\]
 with $f_{\theta,1}=P_\theta\ol\phi f_M$ and $f_{\theta,2}=P_\theta\ol\psi f_M$.
 This in turn is equivalent to
 \beq\label{eq:6aug3}
 P_\theta g f_{\theta,1}+ P_\theta \ol\phi \psi g f_{\theta,2}=0
 \qquad \hbox{and} \qquad
 P_\theta \ol\psi \phi g f_{\theta,1} + P_\theta g  f_{\theta,2}=0.
 \eeq
 Since $f_{\theta,1},f_{\theta,2} \in K_\theta$, we have that
 $f_{\theta,1} = f_{1+}$ and $f_{\theta,2}=f_{2+}$ with $f_{1+},f_{2+} \in H^2_+$ and 
 $\ol\theta f_{1+},\ol\theta f_{2+} \in H^2_-$.
 Thus \eqref{eq:6aug3} is equivalent to 
 
 \begin{equation*}
  \begin{cases}
    \ol\theta f_{1+}= f_{1-} \\
    \ol\theta f_{2+}=f_{2-} \\
    gf_{1+}+g\ol\phi \psi   f_{2+} + \theta f_{3+} = f_{3-} \\
    g\phi\ol\psi f_{1+} + gf_{2+} + \theta f_{4+}= f_{4-},
  \end{cases}
\end{equation*}
with $f_{i\pm} \in H^2_\pm$ for $i=1,2,3,4$,
which, in matrix form, can be written as
\beq\label{eq:8aug3}
\G \Phi_+=\Phi_- \qquad \hbox{with} \qquad \Phi_\pm \in (H^2_\pm)^4,
\eeq
where 
\[
\G = \begin{pmatrix}
\ol\theta & 0 & 0 & 0 \\
0 & \ol\theta & 0 & 0 \\
g & g \ol\phi \psi & \theta & 0 \\
g\phi \ol\psi & g & 0 & \theta
\end{pmatrix}= \begin{pmatrix} \ol\theta I_2 & 0 \\ G & \theta I_2 \end{pmatrix}, 
\]
with 
\[
G= g \begin{pmatrix}1 & \ol\phi\psi \\ \phi\ol\psi & 1 \end{pmatrix}.
\]
The equation \eqref{eq:8aug3} describes the kernel of the block Toeplitz operator $T_\G$.

One can show, as in Section~\ref{sec:3} for TTO, that this Toeplitz operator is 
EAE to the block TTO $A^\theta_G$. Indeed we have:

\begin{thm}
\cite{COP22}
\[
A^M_g \simast A^\theta_G \simast T_\G.
\]
\end{thm}

\section{Conclusion}

We have seen that EAE gives more information than its initial structure would suggest.
Besides its use in the analysis of properties such as invertibility and the Fredholm condition, it has
yielded results showing how the kernel of a (generalized) Toeplitz operator 
can be described starting from the knowledge that certain functions are   contained in
it.

In each case that we have studied,
EAE has arisen starting from a natural isomorphism between kernels. It would be interesting to know for which other classes of operators similar EAE constructions can be found by these methods.

\end{document}